\documentclass[a4paper]{amsart}
\usepackage[utf8]{inputenc}
\usepackage[T1]{fontenc}
\usepackage{lmodern}
\usepackage{amsmath, amsthm, amssymb, amsfonts}
\usepackage[all]{xy}
\usepackage{nicefrac,mathtools}
\usepackage[shortlabels, inline]{enumitem}
\usepackage{microtype}
\usepackage{hyperref}
\usepackage{graphicx}
\usepackage[all]{xy}
\usepackage{xcolor}
\usepackage{tikz-cd}

\usepackage{thmtools}

\numberwithin{equation}{section}
\theoremstyle{plain}
\newtheorem{theorem}[equation]{Theorem}

\newtheorem{lemma}[equation]{Lemma}

\newtheorem{proposition}[equation]{Proposition}

\newtheorem{corollary}[equation]{Corollary}

\theoremstyle{definition}
\newtheorem{definition}[equation]{Definition}

\theoremstyle{remark} 
 \newtheorem*{remark*}{Remark}
\newtheorem*{remarks*}{Remark}

\newcommand*{\Hilm}[1][H]{\mathcal #1}
\newcommand{\Bound}{\mathbb B}

\newcommand{\Comp}{\mathbb{K}}

\newcommand{\norm}[1]{\lvert\!\lvert #1\rvert\!\rvert}
\newcommand{\C}{\mathbb{C}}

\newcommand{\N}{\mathbb{N}}

\newcommand{\Mat}{\mathrm{M}}
\newcommand{\alg}{\textup{alg}}

\newcommand*{\nb}{\nobreakdash}
\newcommand*{\defeq}{\mathrel{\vcentcolon=}}

\title[Haagerup Property]{On the Haagerup property for partial crossed products}

\author{Md Amir Hossain} \email{mdamirhossain18@gmail.com}

\address{Theoretical Statistics and Mathematics Unit, Indian Statistical Institute, Delhi Centre, 7 S. J. S. Sansanwal Marg, New Delhi 110016, India.
}

\author{Chaitanya J. Kulkarni} \email{chaitanyakulkarni58@gmail.com}

\address{MIT Art, Design and Technology University, Pune, 412201, India.
}


\begin{document}
\begin{abstract}
Let $(A,G,\alpha)$ be a partial dynamical system and let 
$A\rtimes_{\alpha,r} G$ denote the associated reduced partial crossed product. 
In this article, we introduce the Haagerup property for partial actions of discrete groups on $C^*$-algebras. 
We prove that the partial crossed product 
$A\rtimes_{\alpha,r} G$ has the Haagerup property if and only if both 
$A$ and the partial action $\alpha$ have the Haagerup property. 
As a consequence, we obtain an equivalence between the Haagerup property of the partial crossed product and that of the underlying $C^*$-algebra and the acting group. 
We also show that the Haagerup property is preserved under inductive limits and apply this result to study the Haagerup property of inductive limits of partial crossed products.
\end{abstract}

\maketitle

\section{Introduction}

A discrete group \(G\) is said to have the Haagerup property if there exists a sequence of positive definite functions 
\(\{\phi_n\}_{n\in\mathbb{N}}\) on \(G\) that vanish at infinity and converge pointwise to \(1\). 
This notion was introduced by Haagerup~\cite{Haagerup-1979-An-example-non-nuclear-alg}, 
who proved that free groups possess this property. 
The Haagerup property is strictly weaker than amenability and is equivalent to Gromov's a-T-menability. 

The Haagerup property has since been extended to operator algebraic settings. 
Choda~\cite{Choda-1983-Group-factor-Haagerup-type} introduced an analogue for von Neumann algebras, 
and Dong~\cite{Dong-2012-Haagerup-prop-C_st-alg} formulated the Haagerup property for unital \(C^*\)\nb-algebras equipped with a faithful tracial state. 
Motivated by these developments, You~\cite{You-group-action-preserving-Haagerup-prop} and Meng~\cite{Meng-Haagerup-prop-for-C-st-alg} investigated the permanence of the Haagerup property under crossed product constructions arising from (global) group actions.

Partial actions of discrete groups on $C^*$-algebras were introduced independently by Exel~\cite{Exel-1994-Circle-action-Cst-alg} and McClanahan~\cite{McClanahan-1995-K-theory-Par-act}. 
They provide a flexible framework that generalizes global actions and give rise to rich classes of $C^*$-algebras. 
For instance, important examples such as Bunce--Deddens algebras and certain AF-algebras can be realized as partial crossed products 
(see, e.g.,~\cite{Exel1994-Bunce-Deddens-alg-as-crossed-product-by-partial-Automorphism, Exel-1995-AF-alg-and-partial-action}). 
Partial crossed products have since become an important tool in the study of $C^*$-algebraic dynamical systems.

The goal of this article is to study the Haagerup property in the setting of partial crossed products. 
Inspired by the work of You~\cite{You-group-action-preserving-Haagerup-prop} for global actions, 
we introduce the notion of the Haagerup property for a partial action. 
We show that if a partial action \(\alpha\) and the underlying \(C^*\)\nb-algebra \(A\) have the Haagerup property, then the reduced partial crossed product \(A\rtimes_{\alpha,r} G\) also has the Haagerup property. Our main result establishes a converse under a mild hypothesis on the partial action. 
More precisely, let \((A,G,\alpha)\) be a \emph{(unital)} partial dynamical system and let \(A\rtimes_{\alpha,r} G\) denote the associated reduced partial crossed product. 
We prove that the following conditions are equivalent:
\begin{enumerate}
	\item The partial crossed product \(A\rtimes_{\alpha,r} G\) has the Haagerup property.
	\item Both \(A\) and \(G\) have the Haagerup property.
	\item The partial action \(\alpha\) has the Haagerup property and \(A\) has the Haagerup property.
\end{enumerate}
This provides a characterization of the Haagerup property for partial crossed products in terms of the underlying algebraic and dynamical data.

In addition, motivated by results of Suzuki~\cite{Suzuki-Haagerup-prop-C_st-alg-Prop-T} on permanence properties of the Haagerup property (such as stability under direct sums and free products), we show that the Haagerup property is preserved under inductive limits of \(C^*\)\nb-algebras. 
As an application, we study the Haagerup property for inductive limits of partial crossed products.

\paragraph{\itshape Structure of the article.}
Section~\ref{prelim} contains preliminaries on partial actions, partial crossed products, and the Haagerup property for \(C^*\)\nb-algebras. 
In Section~\ref{sec-Haagerup-prop}, we introduce the Haagerup property for partial actions and prove the main characterization theorem for partial crossed products (Theorem~\ref{thm-equi-cond-Haagerup-prop}).
In Section~\ref{sec-ind-limit}, we establish the stability of the Haagerup property under inductive limits (Proposition~\ref{prop-ind-limit-Haagerup}) and apply this to inductive limits of partial crossed products (Corollary~\ref{coro-ind-lim-part-Haage}).

\section{Preliminaries}\label{prelim}

\paragraph{Notations} Throughout the article, \(G\) denotes a discrete group with identity element~\(e\). Let \(A\) be a unital \(C^*\)-algebra and \(\Bound(\Hilm)\) denotes the algebra of all bounded linear operators on a separable Hilbert space \(\Hilm\).

\subsection{Partial crossed products}

In this section, we recall the basic definitions and notions of partial actions and associated partial crossed products, from~\cite{Exel2017Book-Partial-action-Fell-bundle} and~\cite{McClanahan-1995-K-theory-Par-act}.

\begin{definition}[{\cite[Definition 11.4]{Exel2017Book-Partial-action-Fell-bundle}}]
	A \emph{partial action} of \(G\) on a \(C^*\)\nb-algebra \(A\) is a pair
	\(\alpha= (\{D_g\}_{g\in G}, \{\alpha_{g}\}_{g\in G})\), consisting of closed two-sided ideals \(\{D_g\}_{g\in G}\) of \(A\) and \(^*\)\nb-isomorphism \(\alpha_{g}\colon D_{g^{-1}} \to D_{g}\) for \(g\in G\) satisfying 
	\begin{enumerate}
		\item \(D_e= A\) and \(\alpha_{e} \colon D_e\to D_e\) is identity;
		\item\label{def-par-act-2}  \(\alpha_{g}\circ \alpha_{h} \subseteq \alpha_{gh}\) for \(g,h\in G\).
	\end{enumerate} 
\end{definition}
\noindent The composition \(\alpha_{g}\circ \alpha_{h}\) is not defined in the usual sense. Its domain is
\[
\{x\in D_{h^{-1}}: \alpha_h(x) \in D_{g^{-1}}\} = \alpha^{-1}_h(D_{g^{-1}}).
\]
 Condition~\eqref{def-par-act-2} in the above definition asserts that
\(\alpha_g(\alpha_h(x)) = \alpha_{gh}(x)\) for \(x\in \textup{dom}(\alpha_g\circ \alpha_h)\).
 In other words, \(\alpha_{gh}\) extends the composition \(\alpha_g\circ \alpha_h\).
We refer to \((A,G, \alpha)\) as a \(C^*\)\nb-algebraic partial dynamical system.

Let \((B, G, \eta)\) be a (global) \(C^*\)\nb-dynamical system and let \(A\) be a closed two-sided ideal of \(B\). Define \(D_{g} =\eta_{g}(A)\cap A\) and \(\alpha_{g}\colon D_{g^{-1}} \to D_{g}\) by the restriction of \(\eta_{g}\) on \(D_{g^{-1}}\). Then \(\alpha= \big(\{D_{g}\}_{g\in G}, \{\alpha_{g}\}_{g\in G} \big)\) is a partial action, called the \emph{restriction partial action}. We say \(\eta\) is a globalization of the partial action \(\alpha\) (see~\cite[Definition 28.1]{Exel2017Book-Partial-action-Fell-bundle}) if 
\[
 B =\overline{\sum_{g\in G}\eta_{g}(A)}.
\]
Note that a globalization of a given partial action need not exist.  However, if it exists, then it is unique (see~\cite[Proposition 28.2]{Exel2017Book-Partial-action-Fell-bundle}). A partial action is said to be \emph{unital} if each ideal \(D_g\) has a unit \(1_g\)  and every \(\alpha_{g}\) is unital. Note that \(1_g\) need not coincide with the unit of the algebra \(A\); in fact, each \(1_g\) is a central idempotent in~\(A\). Moreover, a partial action is unital if and only if it admits a globalization (see~\cite[Theorem 6.13]{Exel2017Book-Partial-action-Fell-bundle}).

The collection of all finite formal sums \(\sum_{g\in G} a_g\delta_g\) with \(a_g\in D_g\) is denoted by \(A\rtimes_{\alg}G\). The multiplication and involution on \(A\rtimes_{\alg} G\) are given by
\[
a\delta_g \cdot b\delta_h \defeq \alpha_g\big(\alpha_{g^{-1}}(a)b\big) \delta_{gh} \quad \textup{and} \quad (a\delta_g)^* \defeq \alpha_{g^{-1}}(a^*)\delta_{g^{-1}}
\]
for \(a\in D_g, b\in D_h\) and \(g,h\in G\) (see \cite[Chapter 8]{Exel2017Book-Partial-action-Fell-bundle}). Let \(\pi\colon A \to \Bound(\Hilm) \) be a representation. For \(g\in G\), define \(\pi_g\colon D_{g} \to \Bound(\Hilm)\) by \(\pi_g(a) =\pi(\alpha_{g^{-1}}(a))\). Then there exists a unique extension of \(\pi^{\prime}_g\) of \(\pi_g\) to \(A\) which annihilates
\[
 \textup{span}\{\pi_g(a)h: a\in D_g, h\in \Hilm\}^{\perp}.
\]
Define \(\widetilde{\pi}\colon A\to \Bound(\ell^2(G,\Hilm))\) by 
\[
\widetilde{\pi}(a)f(g) = \pi^{\prime}_{g}(a)f(g)
\]
where \(f\in \ell^2(G,\Hilm)\) and \(g\in G\). Proposition~3.1 of~\cite{McClanahan-1995-K-theory-Par-act} ensures that the left regular representation \(\lambda\colon  G \to \Bound(\ell^2(G,\Hilm))\) satisfies the covariance condition 
\[
\lambda_g\widetilde{\pi}(a)\lambda_{g^{-1}} =\widetilde{\pi}(\alpha_g(a))
\]
for \(a\in D_{g^{-1}}\) and \(g\in G\). And \(\lambda_g\) is given by 
\(\lambda_g(f)(h) = f(g^{-1}h)\) for \(f\in \ell^2(G,\Hilm)\). The integrated form of \((\widetilde{\pi}, \lambda)\) is given by 
\(\widetilde{\pi}\rtimes\lambda (\sum_{g\in G} a_g\delta_g) = \sum_{g\in G} \pi^{\prime}(a_g)\lambda_g \). The \emph{reduced norm}
 on \(A\rtimes_{\alg}G\) is defined by 
 \[
  \norm{f}_r =\sup\big\{\norm{\widetilde{\pi}\rtimes \lambda (f)} : \pi\colon A\to \Bound(\Hilm) \textup{ is a representation}\big\}.
 \]
 The completion of \(A\rtimes_{\alg}G\) with respect to \(\norm{\cdot}_r\) is called \emph{reduced partial crossed product} and denoted by \(A\rtimes_{\alpha, r}G\). Whenever no confusion arises, we omit the word \emph{reduced} and refer to it simply as the \emph{partial crossed product}.

Consider the partial dynamical system \((A,G,\alpha)\). The map \(i\colon A \to A\rtimes_{\alpha, r}G\) defined by \(a\mapsto a\delta_e\) is an embedding. Moreover, there is a canonical \emph{conditional expectation} \(E\colon A\rtimes_{\alpha, r}G \to A\) by 
\begin{equation}\label{equ-cond-exp}
	E(a\delta_g) = \begin{cases}
		a & \textup{if } g=e,\\
		0 & \textup{if } g\neq e
	\end{cases}
\end{equation}
for \(g\in G\) and \(a\in A_g\). A tracial state \(\tau\) on \(A\) is called \(G\)\nb-invariant if 
\(
\tau(\alpha_{g}(a)) =\tau(a)
\)
for all \(a\in D_{g^{-1}}\).
If \(\tau\) is a \(G\)\nb-invariant tracial state on \(A\), then \(A\rtimes_{\alpha, r}G\) has a tracial state \(\widetilde{\tau} = \tau \circ E\). If \(G\) is \emph{superamenable}, then any tracial state on \(A\) induces a tracial state on \(A\rtimes_{\alpha, r}G\) (see~\cite{Scarparo2017-Supramenable-group-partial-act}).

\subsection{The Haagerup property for \(C^*\)-algebras}

The Haagerup property was originally introduced for groups by Haagerup in~\cite{Haagerup-1979-An-example-non-nuclear-alg} as a notion weaker than amenability. It was subsequently extended to von Neumann algebras by Choda in~\cite{Choda-1983-Group-factor-Haagerup-type}, and later to unital \(C^*\)-algebras by Dong in~\cite{Dong-2012-Haagerup-prop-C_st-alg}. In this section, we briefly recall the Haagerup property for \(C^*\)-algebras. For a more comprehensive treatment, we refer the reader to~\cite{Dong-2012-Haagerup-prop-C_st-alg} and~\cite{Dong-Ruan-Hilbert-mod-Haagerup}.

Let \(A\) be a \(C^*\)-algebra.
A linear map \(\phi\colon A\to A\) is called \emph{completely positive} if \(\phi_n\colon \Mat_n(A) \to \Mat_n(A)\) is positive for all \(n\in\N\), where \(\Mat_n(A)\) is denoted by the collection of all \(n\times n\) matrices on \(A\). We call \(\phi\) is \emph{unital completely positive}, if it is completely positive and unit preserving. Let \(\tau\) be a faithful tracial state on \(A\). A completely positive map \(\phi\) is called \emph{\(\tau\)\nb-decreasing} if \(\tau\circ \phi \leq \tau\). 
A \(\tau\)-decreasing completely positive map \(\phi\) can be extended to a contraction \(\widetilde{\phi}\colon L^2(A,\tau) \to L^2(A,\tau)\) by 
\begin{equation}\label{eq-ind-Hilbert-map}
	 \widetilde{\phi}(a+N_\tau) = \phi(a)+N_\tau
\end{equation}
where \(L^2(A,\tau)\) is the GNS Hilbert space associated to \(\tau\) and \(N_{\tau} =\{a\in A: \tau(a^*a) =0\}\). We write \(\norm{\cdot}_{2,\tau}\) for the Hilbert space norm on \(L^2(A, \tau)\), in order to emphasize its dependence on the tracial state \(\tau\) and to distinguish it from the \(C^*\)\nb-algebra norm on~\(A\). We say \(\phi\) is \(L_2\)\nb-compact if \(\widetilde{\phi}\) is a compact operator on \(L^2(A,\tau)\). Using finite-rank operator, one can see that \(\phi\) is \(L_2\)\nb-compact if and only if for \(\epsilon>0\), there exists a finite-rank operator \(R\colon A\to A\)\footnote{The operator \(R\) can be viewed as a map from \(L^2(A,\tau) \to L^2(A,\tau)\) given by \(a+N_{\tau} \mapsto R(a)+N_{\tau}\).} such that 
\[
 \norm{\phi(a)-R(a)}_{2,\tau} \leq \epsilon \norm{a}_{2,\tau}
\] 
for all \(a\in A\).

\begin{definition}[{\cite[Definition 2.3]{Dong-2012-Haagerup-prop-C_st-alg}}]\label{def-Haagerup-prop}
Let \(A\) be a unital \(C^*\)-algebra with a faithful tracial state \(\tau\). Then \((A, \tau)\) is said to have the \emph{Haagerup property} if there exists a \(\tau\)\nb-decreasing sequence of unital completely positive map \(\{\phi_n\}_{n\in\N}\) on \(A\) such that 
\begin{enumerate}
	\item\label{def-cond-1} \(\phi_n\) is \(L_2\)\nb-compact for all \(n\in \N\);
	\item\label{def-cond-2} \(\norm{\phi_n(a)-a}_{2,\tau}\to 0\) as \(n\to \infty\) for all \(a\in A\).
\end{enumerate}
\end{definition}

\section{Haagerup property for partial crossed products}\label{sec-Haagerup-prop}

Let \((A,G, \alpha)\) be a partial dynamical system. For a map \(h\colon G \to A\) and \(g_1,g_2,\cdots, g_n \in G\), we define 
\[
 a_{i,j} \defeq \begin{cases}
 	\alpha_{g_j}(h(g_{i}^{-1}g_j)) & \textup{ if } h(g_{i}^{-1}g_j) \in D_{g_{j}^{-1}},\\
 	0 & \textup{ otherwise}
 		\end{cases}
\] 
where \(i,j=1,2,\cdots, n\). We call the map \(h\) is positive definite with respect to the partial action \(\alpha\) if the matrix \((a_{i,j})_{1\leq i,j\leq n}\geq 0\) in \(\Mat_{n}(A)\) for any \(g_1,g_2,\cdots, g_n \in G\).

Recall from~\cite{You-group-action-preserving-Haagerup-prop}, a map \(h\colon G\to A\) is vanishing at infinity with respect to a faithful tracial state \(\tau\) on \(A\) if for any \(\epsilon>0\) there exist a finite set \(F\) of \(G\) such that \(\norm{h(g)}_{2,\tau}<\epsilon\) for all \(g\notin F\), where \(\norm{\cdot}_{2,\tau}\) is the norm on the GNS Hilbert space \(L^2(A,\tau)\). We denote this by \(h\in C_{0,\tau}(G,A)\).

Motivated from an amenable action of a discrete group on a \(C^*\)-algebra (see~\cite[Definition 4.3.1]{Brown-Ozawa2008Cst-alg-and-finite-dim-approximations}), You in~\cite[Definition 1.3]{You-group-action-preserving-Haagerup-prop} defined the Haagerup property for a (global) action of groups on \(C^*\)-algebras and proved an analogous result as~\cite[Therorem 4.3.4]{Brown-Ozawa2008Cst-alg-and-finite-dim-approximations}, that is, if a (global) action has Haagerup property and the \(C^*\)-algebra has Haagerup property, then the associated crossed product also has the Haagerup property. Motivated from this result, we now introduce the Haagerup property for partial actions on \(C^*\)-algebras.

\begin{definition}\label{def-par-act-pres-Haagerup}
	Let \((A,G, \alpha)\) be a partial dynamical system. We call the partial action \(\alpha\) has the \emph{Haagerup property} with respect to a faithful tracial state \(\tau\) on \(A\) if there exists a sequence of bounded positive definite functions \(\{h_n\colon G\to Z(A)\}_{n\in \N} \subseteq C_{0,\tau}(G,A)\) such that \(h_n(g)\in D_g\) for \(g\in G\) and \(h_n\to 1\) pointwise with respect to \(\tau\), that is, \(\norm{h_n(g)-1}_{2,\tau}\to 0\) as \(n\to \infty\) for \(g\in G\).
 \end{definition}
 
 Recall the definition of globalization of a partial action from Section~\ref{prelim}. Suppose \(\eta\) is the globalization of a partial action \(\alpha\) of \(G\) on \(A\). Then \(\alpha\) has the Haagerup property with respect to a tracial state \(\tau\) on \(A\) if and only if the globalization \(\eta\) has the Haagerup property in the sense of You~(\cite[Definition 1.3]{You-group-action-preserving-Haagerup-prop}).

 \begin{lemma}\label{lem-positive-def-UCP}
 Let \((A,G, \alpha)\) be a partial dynamical system and let \(h\colon G\to Z(A)\) be a positive definite map (with respect to \(\alpha\)) satisfying \(h(g)\in D_g\) for \(g\in G\). If \(\phi\colon A\to A\) is a unital completely positive map, then the map \(\Phi\colon A\rtimes_{\alpha, r}G\to A\rtimes_{\alpha, r} G\) defined by 
 \begin{equation}\label{eq-induced-map-crossed-prod}
 \Phi\Big(\sum_{g\in G} a_g\delta_g\Big) = \sum_{g\in G}\phi(a_g)h(g)\delta_g
  \end{equation}
 is unital completely positive.
 \end{lemma}
 \noindent The proof is analogous to that of~\cite[Theorem 3.2]{Dong-Ruan-Hilbert-mod-Haagerup} in the case of global actions, and we therefore omit the details.

\begin{lemma}\label{lem-L-2-compact}
Let \((A,G, \alpha)\) be a partial dynamical system with a faithful \(G\)\nb-invariant tracial state \(\tau\) on \(A\) and let \(h\colon G\to Z(A)\) be a positive definite map (with respect to \(\alpha\)) satisfying \(h(g)\in D_g\) for \(g\in G\). Suppose \(\phi\colon A\to A\) is a \(\tau\)\nb-decreasing unital completely positive map which is \(L_2\)\nb-compact. Let \(\Phi\) be the unital completely positive map on \(A\rtimes_{\alpha, r}G\) defined by Equation~\eqref{eq-induced-map-crossed-prod}. 
If \(h\) vanishes at infinity with respect to~\(\tau\), then the induced operator \(\widetilde{\Phi}\) on \(L^2(A\rtimes_{\alpha, r}G, \widetilde{\tau})\)
is compact (recall that \(\widetilde{\tau} =\tau\circ E\), where \(E\) is given by Equation~\eqref{equ-cond-exp}).
\end{lemma}
\begin{proof}
Since \(h\) is vanishing at infinity with respect to \(\tau\), for each \(n\in \N\) there exists a finite set \(F_n\subset G\) such that \(\norm{h(g)}_{2,\tau}<1/n\) for \(g\notin F_n\). Since \(\phi\) is \(L_2\)-compact there exists a sequence of finite-rank operators \(\{\phi_n\}_{n\in \N}\) on \(A\) such that 
\(\norm{\widetilde{\phi}-\widetilde{\phi}_n}_{2,\tau}<1/n\). Define \(T_n\) on \(A\rtimes_{\alpha, r}G\) by 
\[
T_n\Big(\sum_{g\in G} a_g\delta_g\Big) = \sum_{g\in F_n}\phi_n(a_g)h(g)\delta_{g}
\]
where \(\sum_{g\in G} a_g\delta_g \in A\rtimes_{\alg}G\).
The induced map \(\widetilde{T}_n\) on \(L^2(A\rtimes_{\alpha, r}G, \widetilde{\tau})\) is finite-rank as \(\phi_n\) is so. For \(x=\sum_{g\in G}a_g\delta_g \in A\rtimes_{\alg}G\), we have 
\[
\Phi(x)-T_n(x) = \sum_{g\in F_n} (\phi-\phi_n)(a_g)h(g)\delta_g +\sum_{g\notin F_n}\phi(a_g)h(g)\delta_g.
\]
Denote these two sums by \(M_1\) and \(M_2\), respectively and using the inequality \(\norm{\widetilde{\Phi}(x)-\widetilde{T}_n(x)}^2_{2,\tilde{\tau}} \leq 2(\norm{M_1}^2_{2, \tilde{\tau}} + \norm{M_2}^2_{2, \tilde{\tau}})\), we estimate each term separately.
For \(M_1\), we compute 
\begin{align*}
\norm{M_1}^2_{2,\tilde{\tau}} &=\tau\Big(\sum_{g\in F_n} h^*(g)(\phi-\phi_n)^*(a_g)(\phi-\phi_n)(a_g)h(g)\Big) \\
&= \sum_{g\in F_n} \tau\big( h^*(g)(\phi-\phi_n)^*(a_g)(\phi-\phi_n)(a_g)h(g) \big) \\
&\leq \sum_{g\in F_n}\tau\big((\phi-\phi_n)^*(a_g)(\phi-\phi_n)(a_g)\big) \tau\big( h^*(g)h(g)\big)\\
&= \sum_{g\in F_n} \norm{(\widetilde{\phi}-\widetilde{\phi}_n)(a_g)}^2_{2,\tau} \norm{h(g)}^2_{2,\tau}. 
\end{align*}
Since \(\norm{\widetilde{\phi}-\widetilde{\phi}_n}_{2,\tau}<1/n\) and \(h\) is bounded, the last term above dominated by 
\[
\sum_{g\in F_n} \frac{1}{n^2}\norm{a_g}^2_{2,\tau} K \leq \frac{K}{n^2}\sum_{g\in F_n}\tau(a^*_ga_g) = \frac{K}{n^2}\norm{x}^2_{2,\tilde{\tau}}
\]
where \(K=\sup_{g\in G}\norm{h(g)}^2_{2,\tau}\). For \(M_2\), using \(\norm{h(g)}_{2, \tau} \leq \frac{1}{n}\) for \(g\notin F\) and boundedness of \(\widetilde{\phi}\), and a similar computation as above gives us 
\[
\norm{M_2}^2_{2,\tilde{\tau}} \leq \sum_{g\notin F_n}\frac{1}{n^2} \norm{\widetilde{\phi}}^2_{2,\tau} \norm{a_g}^2_{2,\tau} \leq \frac{1}{n^2} \norm{\widetilde{\phi}}^2_{2,\tau} \norm{x}^2_{2,\tilde{\tau}}.
\]
Therefore, \(\norm{\widetilde{\Phi}(x)-\widetilde{T}_n(x)}^2_{2,\tilde{\tau}} \leq \frac{1}{n}\big(2(K+\norm{\widetilde{\phi}})\big)^{\frac{1}{2}}\norm{x}^2_{2,\tilde{\tau}}\). Thus, \(\widetilde{\Phi}\) is compact.
\end{proof}

 \begin{proposition}\label{prop-partial-crossed-prod-Haagerup}
 Let \((A,G, \alpha)\) be a partial dynamical system such that \(\alpha\) has the Haagerup property with respect to a \(G\)\nb-invariant faithful tracial state \(\tau\) on \(A\). If \((A, \tau)\) has the Haagerup property, then the partial crossed product \((A\rtimes_{\alpha, r}G, \widetilde{\tau})\) also has the Haagerup property.
 \end{proposition}
 \begin{proof}
Since \((A, \tau)\) has the Haagerup property, there exists a \(\tau\)\nb-decreasing sequence of unital completely positive maps \(\{\phi_n\}_{n\in \N}\) on \(A\) satisfying Conditions~\eqref{def-cond-1} and~\eqref{def-cond-2} of Definition~\ref{def-Haagerup-prop}. The Haagerup property of \(\alpha\) gives us a sequence of bounded positive definite functions \(\{h_n\colon G\to Z(A)\}_{n\in \N} \subseteq C_{0,\tau}(G,A)\) such that \(h_n(g)\in D_g\) for \(g\in G\) and \(h_n\to 1\) pointwise with respect to \(\tau\). For \(n\in \N\), define a completely positive map~\(\Phi_n\) on \(A\rtimes_{\alpha, r}G\) corresponding to \(\phi_n\) and \(h_n\) as in Equation~\eqref{eq-induced-map-crossed-prod}. The sequence \(\{\Phi_n\}_{n\in \N}\) is \(\widetilde{\tau}\)-decreasing as \(\{\phi_n\}_{n\in \N}\) is \(\tau\)\nb-decreasing. 
Lemma~\ref{lem-L-2-compact} ensures that \(\Phi_n\) is \(L_2\)\nb-compact for all \(n\in \N\). Let \(f=\sum_{g\in F}a_g\delta_g \in A\rtimes_{\alg}G\) for some finite set \(F\subset G\). Then
\begin{align*}
	\norm{\Phi_n(f)-f}^2_{2,\tilde{\tau}} &=\big\|\sum_{g\in F}\phi_n(a_g)h_n(g)\delta_g-a_g\delta_{g}\big\|^2_{2,\tilde{\tau}}\\
	&= \tau\Big(\sum_{g\in F} \big(\phi_n(a_g)h_n(g)-a_g\big)^*\big(\phi_n(a_g)h_n(g)-a_g\big)\Big)\\
	&=\sum_{g\in F}\norm{\phi_n(a_g)h_n(g)-a_g}^2_{2, \tau}\\
	&\leq 2 \sum_{g\in F} \bigr(\norm{\phi_n(a_g)h_n(g)-\phi_n(a_g)}^2_{2, \tau} + \norm{\phi_n(a_g)-a_g}^2_{2, \tau}\bigr)\\
	&\leq 2 \sum_{g\in F} \bigr(\norm{\phi_n(a_g)}_{2,\tau}\norm{h_n(g)-1}^2_{2, \tau} + \norm{\phi_n(a_g)-a_g}^2_{2, \tau}\bigr)\\
	&\leq 2 \sum_{g\in F} K\norm{h_n(g)-1}^2_{2, \tau} + \norm{\phi_n(a_g)-a_g}^2_{2, \tau}\bigr) \to 0\quad  \textup{as } n\to \infty
\end{align*}
where \(K=\sup_{g\in F}\norm{\phi_n(a_g)}_{2,\tau}\leq \sup_{g\in F} \norm{a_g}_{2, \tau}\) as each \(\phi_n\) is unital. Therefore, \(\Phi_n\) converges pointwise to the identity map on \(L^2(A\rtimes_{\alpha, r}G, \widetilde{\tau})\). This completes the proof.
\end{proof}
 
 \begin{lemma}\label{lem-Haagerup-prop-for-alg-A}
 	Let \((A,G, \alpha)\) be a partial action with a \(G\)\nb-invariant faithful tracial state \(\tau\) on \(A\). If \((A\rtimes_{\alpha, r}G, \widetilde{\tau})\) has the Haagerup property, then \((A, \tau)\) has the Haagerup property.
 \end{lemma}
\begin{proof}
Since \((A\rtimes_{\alpha, r}G, \widetilde{\tau})\) has the Haagerup property, there exist a sequence \(\{\Phi_n\}_{n\in \N}\) of unital completely positive maps on \(A\rtimes_{\alpha, r}G\) satisfying the conditions of Definition~\ref{def-Haagerup-prop}. Recall the conditional expectation \(E\colon A\rtimes_{\alpha, r}G\to A\) from Equation~\eqref{equ-cond-exp}. For \(n\in \N\), define unital completely positive map on \(A\) by
\[
\phi_n(a)=E\circ \Phi_n(a) \quad \textup{for } a\in A.
\]
Then \(\tau\circ \phi_n \leq \tau\) as \(\widetilde{\tau}\circ \Phi_n \leq \widetilde{\tau}\). Let \(\epsilon >0\). Since \(\Phi_n\) is \(L_2\)\nb-compact, there exist a finite-rank operator \(R_n\) on \(A\rtimes_{\alpha, r}G\) such that 
\[
\norm{\Phi_n(f)-R_n(f)}_{2,\tilde{\tau}} \leq \epsilon \norm{f}_{2, \tilde{\tau}}
\] 
for \(f\in A\rtimes_{\alpha, r}G\). Now \(E\circ R_n\) is also a finite-rank operator on \(A\) and we have
\begin{multline*}
\norm{\phi_n(a)-E\circ R_n(a)}_{2,\tau} = \norm{E\circ\Phi_n(a)-E\circ R_n(a)}_{2,\tau} \\ \leq \norm{\Phi_n(a)- R_n(a)}_{2,\tilde{\tau}}\leq \epsilon \norm{a}_{2, \tilde{\tau}} = \epsilon\norm{a}_{2,\tau} \quad \textup{for } a\in A.  
\end{multline*}

For \(a\in A\), we have 
\begin{align*}
	\norm{\phi_n(a)-a}^2_{2,\tau} &= \tau\big( (\phi_n(a)-a)^*(\phi_n(a)-a)\big) \\
	&= \tau\big( (E(\phi_n(a)-a))^*(E(\phi_n(a)-a))\big)\\
&\leq \tau\circ E\big( (\phi_n(a)-a)^*(\phi_n(a)-a) \big) \\
&= \norm{\Phi_n(a)-a}^2_{2,\tilde{\tau}} \to 0 \quad \textup{as } n\to \infty. 
\end{align*}
Therefore, \((A, \tau)\) has the Haagerup property.
\end{proof}
\begin{theorem}\label{thm-equi-cond-Haagerup-prop}
Let \(\alpha=(\{D_g\}_{g\in G}, \{\alpha_g\}_{g\in G})\) be a unital partial action of \(G\) on \(A\), that is, each \(D_g\) is unital and \(\tau\) is a \(G\)\nb-invaraint tracial state on \(A\). Then the following statements are equivalent.
\begin{enumerate}
	\item The partial crossed product \((A\rtimes_{\alpha, r}G, \widetilde{\tau})\) has the Haagerup property.
	\item Both \((A, \tau)\) and \(G\) have the Haagerup property.
	\item The partial action \(\alpha\) has the Haagerup property and \((A,\tau)\) has the Haagerup property.
\end{enumerate}
\end{theorem}
\begin{proof}
\noindent (1).\(\implies\)(2). Lemma~\ref{lem-Haagerup-prop-for-alg-A} ensures that \((A, \tau)\) has the Haagerup property.

Since \((A\rtimes_{\alpha, r}G, \widetilde{\tau})\) has the Haagerup property there exist a sequence \(\{\Phi_n\}_{n\in \N}\) of unital completely positive maps on \(A\rtimes_{\alpha, r}G\) satisfying the conditions of Definition~\ref{def-Haagerup-prop}.
 To prove \(G\) has the Haagerup property, we define \(\eta_n\colon G\to \C\) by 
\[
 \eta_n(g) = \widetilde{\tau}\big(\Phi_n(1_g\delta_g)(1_{g}\delta_{g})^*\big)
\]
where \(1_g\) is the identity element of \(D_g\), \(g\in G\) and \(n\in \N\). Now \(\eta_n(e) =\widetilde{\tau}\big(\Phi_n(1_e\delta_e)1_e\delta_e\big) =1 \). For \(g_1,g_2,\cdots,g_k\in G\) and \(c_1,c_2,\cdots, c_k\in \C\), we have
 
\begin{align*}
	\sum_{i, j=1}^{k} c_{i}\overline{c_j}\eta_n(g_j^{-1}g_i)&= \sum_{i, j=1}^{k}  c_{i}\overline{c_j}\widetilde{\tau}\Big(\Phi_n(1_{g_j^{-1}g_i}\delta_{g_{j}^{-1}g_i}) (1_{g_j^{-1}g_i}\delta_{g_j^{-1}g_i})^*\Big)\\
	&= \sum_{i, j=1}^{k}  c_{i}\overline{c_j}\widetilde{\tau}\Big(\Phi_n (1_{g_j^{-1}g_i}\delta_{g_{j}^{-1}g_i})  (1_{g_i^{-1}g_j}\delta_{g_{i}^{-1}g_j})\Big) \\
	&=\sum_{i, j=1}^{k}  \widetilde{\tau}\Big(c_i\overline{c_j} \Phi_n\big((1_{g_j^{-1}}\delta_{g_{j}^{-1}})(1_{g_i}\delta_{g_i})\big) (1_{g_i^{-1}}\delta_{g^{-1}_i}1_{g_j}\delta_{g_j})\Big)\\
	&= \sum_{i, j=1}^{k}  \widetilde{\tau}\Big(\overline{c_j}(1_{g_j}\delta_{g_j}) \Phi_n\big((1_{g_j^{-1}}\delta_{g_{j}^{-1}})(1_{g_i}\delta_{g_i})\big) (c_{i}1_{g_i^{-1}}\delta_{g^{-1}_i})\Big)\\
	&= \widetilde{\tau}\Big(\sum_{i, j=1}^{k}  \big(\overline{c_j}(1_{g_j}\delta_{g_j}) \Phi_n\big((1_{g_j^{-1}}\delta_{g_{j}^{-1}})(1_{g_i}\delta_{g_i})\big) (c_{i}1_{g_i^{-1}}\delta_{g^{-1}_i})\Big) \geq 0.
\end{align*} 
 Therefore, \(\eta_n\) is positive definite on \(G\). For \(g\in G\),
 \begin{align*}
 	 |\eta_n(g)-1| &= |\widetilde{\tau}\big(\Phi_n(1_g\delta_g)1_{g^{-1}}\delta_{g^{-1}}\big) -1| \\
 	 &= |\widetilde{\tau}\big(\Phi_n(1_g\delta_g)1_{g^{-1}}\delta_{g^{-1}}\big) -\widetilde{\tau}(1_g\delta_g\cdot 1_{g^{-1}}\delta_{g^{-1}})| \\
 	&= |\widetilde{\tau}\bigr(\big(\Phi_n(1_g\delta_g) - 1_g\delta_g\big) 1_{g^{-1}}\delta_{g^{-1}}\bigr)| \\
 	&\leq  |\widetilde{\tau}\big(\Phi_n(1_g\delta_g) - 1_g\delta_g\big)|\\
 	& \leq \norm{\Phi_n(1_g\delta_g) - 1_g\delta_g}_{2,\tilde{\tau}} \to 0 \quad \textup{as } n\to 0.
 \end{align*} 
  Therefore, \(\eta_n\to 1\) pointwise. Moreover, \(\eta_n\) is vanishing at infinity as \(\Phi_n\) is \(L_2\)\nb-compact. Hence, \(G\) has the Haagerup property.

\noindent (2).\(\implies\)(3). Since \(G\)
 has the Haagerup property, there exist a sequence of positive linear functions \(\{\eta_n\}_{n\in \N}\) such that \(\eta_n(e)=1\), \(\eta_n\) is vanishing at infinity for all \(n \) and \(\eta_n\to 1\) pointwise. Let \(1_g\) be the unit element of \(D_{g}\) for \(g\in G\). For \(n\in \N\), we define \(h_n(g)=\eta_n(g)1_g\). Then the sequence \(\{h_n\}_{n\in\N}\) will fulfill the requirement of Definition~\ref{def-par-act-pres-Haagerup} for the partial action \(\alpha\). Therefore, \(\alpha\) has the Haagerup property.
  
\noindent (3).\(\implies\)(1). Follows from Proposition~\ref{prop-partial-crossed-prod-Haagerup}.
\end{proof}

\section{Haagerup property for inductive limits of \(C^*\)-algebras}\label{sec-ind-limit} 

In this section, we show that the Haagerup property preserves under inductive limits and apply this result to study the Haagerup property for inductive limits of partial crossed products.

Let  \((A^{(n)}, \phi_n)_{n\in \N}\) be an inductive system in a category \(\mathcal{C}\). An \emph{inductive limit} of this system is a pair \((A, \mu_n)\), where \(A\) is an object of \(\mathcal{C}\) and each \(\mu_{n}\colon A^{(n)} \to A\) is a morphism in \(\mathcal{C}\) satisfying \(\mu_{n+1}\circ \phi_n = \mu_{n}\) for \(n\in \N\).
Moreover, the inductive limit is unique up to a canonical isomorphism: if \((B,\lambda_n)\) is another limit of the same system, then there exists a unique morphism, \(\lambda \colon A \to B\) such that \(\lambda \circ \mu_n = \lambda_n\) for \(n \in \N\). Inductive limits may not always exist in a category. For example, in the category of finite sets the inductive limit do not exists. However, in the category of \(C^*\)-algebras inductive limits always exists (see~\cite[Proposition 6.2.4]{Rordam2002K-theory-book}).

 \begin{lemma}\label{lem-identification-Hilbert-spac}
 	Let \( (A^{(n)},\phi_n)_{n\in \mathbb{N}} \) be an inductive system of unital \(C^*\)\nb-algebras, and let \(\{\tau_n\}_{n\in\mathbb{N}}\) be faithful tracial states such that \(	\tau_{n+1}\circ \phi_n=\tau_n\) for all \(n\in \N\).
 	Let \((A,\tau)=\varinjlim (A^{(n)},\tau_n)\), where \(\tau\) is the unique tracial state on \(A\) satisfying \(\tau\circ \lambda_n=\tau_n\) for \(n\in \N\), where \(\lambda_n\colon A_n\to A\) is the canonical embedding. Then there is a canonical unitary
 	\[
 	L^2(A,\tau)\;\cong\;\varinjlim L^2(A^{(n)},\tau_n),
 	\]
 	where the Hilbert space inductive limit is taken with respect to the isometric maps
 	\[
 	U_n\colon L^2(A^{(n)},\tau_n)\to L^2(A^{(n+1)},\tau_{n+1}),
 	\quad \textup{by} \quad
 	U_n(a+N_{\tau_n})=\phi_n(a)+N_{\tau_{n+1}}.
 	\]
 \end{lemma}
 \begin{proof}
 	Since \(\tau_{n+1}\circ\phi_n=\tau_n\), each map \(U_n\) is a well-defined isometry:
 	\[
 	\|U_n(a+N_{\tau_n})\|_{2,\tau_{n+1}}^2
 	= \tau_{n+1}(\phi_n(a)^*\phi_n(a))
 	= \tau_n(a^*a)
 	= \|a+N_{\tau_n}\|_{2,\tau_n}^2.
 	\]
 	Hence we may identify each \(L^2(A^{(n)},\tau_n)\) as a subspace of \(L^2(A^{(n+1)},\tau_{n+1})\). Set
 	\[
 	\Hilm_0=\bigcup_{n\geq 1} L^2(A^{(n)},\tau_n),
 	\]
 	and let \(\Hilm\) be the Hilbert space completion of \(\Hilm_0\). Then \(\Hilm=\varinjlim L^2(A^{(n)},\tau_n)\). Let \(\lambda_n\colon A^{(n)}\to A\) denote the canonical embedding of the inductive limit. Define a linear map \(\pi_0 \colon \bigcup_{n\geq 1}\lambda_n(A^{(n)})\to \Hilm_0\)
 	by
 	\[
 	\pi_0(\lambda_n(a))=a+N_{\tau_n}\in L^2(A^{(n)},\tau_n)\subseteq \Hilm_0.
 	\]
 	If \(\lambda_n(a)=\lambda_m(b)\) in \(A\), then choose some \(k\geq n,m\) with
 	\(\phi_{n,k}(a)=\phi_{m,k}(b)\), where \(\phi_{n,k}\colon A^{(n)} \to A^{(k)}\) given by \(\phi_{n,k}=\phi_k\circ\phi_{k-1}\circ \cdots\circ \phi_{n+1}\circ\phi_n\). Similarly, one can define \(\phi_{m,k}\). Analogously, we can define \(U_{n,k}\) and \(U_{m,k}\). Since
 	\[
 	U_{n,k}(a+N_{\tau_n})
 	=\phi_{n,k}(a)+N_{\tau_k}
 	=\phi_{m,k}(b)+N_{\tau_k}
 	=U_{m,k}(b+N_{\tau_m}),
 	\]
 	the map \(\pi_0\) is well-defined. Moreover,
 	\[
     \norm{\pi_0(\lambda_n(a))}_{2,\tau_n}^2
 	=\tau_n(a^*a)
 	=\tau(\lambda_n(a)^*\lambda_n(a)),
 	\]
 	so \( \pi_0 \) preserves the inner product induced by \(\tau\).
 	Thus \(\pi_0\) extends uniquely to an isometry \(\pi \colon L^2(A,\tau)\to \Hilm\).
 	Since \(\bigcup_{n\geq 1}\lambda_n(A^{(n)})\) is dense in \(A\) and \(\Hilm_0\) is dense in \(\Hilm\), the range of \(\pi\) is dense in \(\Hilm\). Hence, \(\pi\) is unitary.
 \end{proof}

\begin{lemma}\label{lem-L2-compact-limit}
	Let \((A^{(n)},\phi_n)_{n\in\mathbb{N}}\) be an inductive system of unital \(C^*\)\nb-algebras with faithful tracial states 
	\(\{\tau_n\}_{n\in\mathbb{N}}\) satisfying \(\tau_{n+1}\circ \phi_n = \tau_n\) for all \(n\in \N\).
	Let \((A,\tau)=\varinjlim (A^{(n)},\tau_n)\), and identify \(L^2(A,\tau)\cong \varinjlim L^2(A^{(n)},\tau_n)\)
	as in Lemma~\ref{lem-identification-Hilbert-spac}. 
	Fix \(n\in\mathbb{N}\), and identify \(L^2(A^{(n)},\tau_n)\) as a closed subspace of
	\(L^2(A,\tau)\). If \(T\in \Comp(L^2(A^{(n)},\tau_n))\)
	is compact, then the operator \(T^{\prime} \colon L^2(A,\tau)\to L^2(A,\tau)\) by
	\[
	T^{\prime}(\xi)=
	\begin{cases}
		T(\xi), & \textup{ if }\xi\in L^2(A^{(n)},\tau_n),\\
		0, & \textup{ if }\xi\in L^2(A^{(n)},\tau_n)^\perp,
	\end{cases}
	\]
	is also compact.
\end{lemma}

\begin{proof}
	Let \(\Hilm=L^2(A,\tau)\) and \(\Hilm_n=L^2(A^{(n)},\tau_n)\), viewed as a closed subspace of \(\Hilm\). Let \(P_n\) denote the orthogonal projection of \(\Hilm\)
	onto \(\Hilm_n\). Then the operator \(T^{\prime}\) can be written as \(T^{\prime} = T P_n\). Thus, \(T^{\prime}\) acts on \(\Hilm_n\) as \(T\) and zero on \(\Hilm_n^{\perp}\). Since \(T\) is compact on the Hilbert space \(\Hilm_n\), there exists a
	sequence of finite-rank operators \(T_m\) on \(\Hilm_n\) such that
	\[
	\norm{T-T_m} \to 0 \quad \textup{as} \quad m\to \infty.
	\]
	Define \(T^{\prime}_m = T_m P_n\) on \(\Hilm\). Then each \(T^{\prime}_m\) has finite-rank in \(\Hilm\),
	because its range is contained in the finite-dimensional subspace
	\(\textup{rang}(T_m)\subseteq H_n\).
	Moreover,
	\[
	\norm{T^{\prime}-T^{\prime}_m}
	= \norm{(T-T_m)P_n}
	\leq \norm{T-T_m}  \to 0 \quad \textup{as} \quad m\to \infty.
	\]
	Hence \(T^{\prime}\) is compact.
\end{proof}

\begin{proposition}\label{prop-ind-limit-Haagerup}
	Let \((A^{(n)},\phi_n)_{n\in \N}\) be an inductive system of unital \(C^*\)-algebras with a sequence of faithful tracial states \(\{\tau_n\}_{n\in \N}\) satisfying \(\tau_{n+1}\circ \phi_n = \tau_n\) for all \(n \in \N\). Let \((A, \tau) = \varinjlim (A^{(n)}, \tau_n)\). If each \((A^{(n)}, \tau_n)\) has the Haagerup property, then the inductive limit \((A, \tau)\) has the Haagerup property.
\end{proposition}
\begin{proof}
Since each \((A^{(n)}, \tau_n)\) has the Haagerup property, there exists a sequence \(\{\varphi_{n,k}\}_{k\in \N}\) of unital completely positive maps on \(A^{(n)}\) such that \(\tau_n\circ \varphi_{n,k}\leq \tau_n\) for all \(k\in \N\). Moreover, \(\varphi_{n,k}\) is \(L_2\)\nb-compact for all \(k\in \N\) and \(\norm{\varphi_{n,k}(a)-a}_{2,\tau_n}\to 0\) as \(k\to \infty\) for all \(a\in A^{(n)}\). We extend \(\varphi_{n,k}\) to a unital completely positive map \(\psi_{n,k}\) on \(A\) satisfying \(\psi_{n,k}(\lambda_n(a)) = \lambda_n(\varphi_{n,k}(a))\) for all \(a\in A^{(n)}\). Then \(\psi_{n,k}\) is \(\tau\)-preserving, that is, \(\tau\circ \psi_{n,k}\leq \tau\). Set  \(\psi_n =\psi_{n,n}\) for \(n\in \N\). Then Lemma~\ref{lem-L2-compact-limit} ensures that \(\psi_n\) is \(L_2\)\nb-compact for all \(n\in \N\). Let \(x=\lambda_n(a)\in \bigcup_{m}\lambda_m(A^{(m)})\), where \(a\in A^{(n)}\). Then
\begin{multline*}
	\norm{\psi_n(x)-x}_{2,\tau} = \norm{\psi_{n,n}(\lambda_n(a))-\lambda_n(a)}_{2, \tau}  = \norm{\lambda_n(\varphi_{n,n}(a))-\lambda_n(a)}_{2, \tau} \\
	= \norm{\varphi_{n,n}(a)-a}_{2, \tau_n} \to 0 \quad \textup{as } n\to \infty. 
\end{multline*}  
After passing to the limit, we conclude that \(\norm{\psi_n(x)-x}_{2,\tau} \to 0\) as \(n\to \infty\) for all \(x\in A=\varinjlim A^{(n)}\). 
Therefore, \((A, \tau)\) has the Haagerup property.
\end{proof}
Recall from~\cite[Definition 3.2]{Hossain-Inductive-limits-partial}, that an inductive system of partial dynamical system is \(\big((A^{(n)}, G, \alpha^{(n)}), \phi_n\big)_{n\in \N}\) where \((A^{(n)}, \phi_n)_{n\in \N}\) is an inductive system of \(C^*\)\nb-algebras and \(\phi_n\colon A^{(n)} \to A^{(n+1)}\) is \(G\)\nb-equivariant for \(n\in \N\). Proposition~3.3 of~\cite{Hossain-Inductive-limits-partial} says that if \(\big((A^{(n)}, G, \alpha^{(n)}), \phi_n\big)_{n\in \N}\) is an inductive system of partial dynamical system, then there exist a unique partial action \(\alpha\) of \(G\) on \(A =\varinjlim A^{(n)}\) such that \(\alpha_g = \varinjlim \alpha_g^{(n)}\) for all \(g\in G\). We call \(\alpha\) as the \emph{inductive limit partial action}.
\begin{corollary}\label{coro-ind-lim-part-Haage}
	Let \(\big((A^{(n)}, G, \alpha^{(n)}), \phi_n  \big)_{n\in \N}\) be an inductive system of partial dynamical system with a sequence of  \(G\)\nb-invariant faithful tracial states \(\{\tau_n\}_{n\in \N}\) satisfying \(\tau_{n+1}\circ \phi_n = \tau_n\). Assume \(G\) is an amenable group. If each \(\alpha_n\) has the Haagerup property and each \((A^{(n)}, \tau_n)\) has the Haagerup property, then the partial crossed product \(A\rtimes_{\alpha,r}G\) has the Haagerup property, where \(\alpha\) is the inductive limit partial action and \(A =\varinjlim A^{(n)}\).
\end{corollary}
\begin{proof}
	Since \(G\) is amenable, by~\cite[Theorem 3.6]{Hossain-Inductive-limits-partial}, we have 
	\[
	 A\rtimes_{\alpha, r} G\cong \varinjlim A^{(n)}\rtimes_{\alpha^{(n)}, r}G.
	\]
	As \(\{\tau_n\}_{n\in \N}\) is a \(G\)\nb-invariant sequence of tracial state, \(A\) has a \(G\)\nb-invariant tracial state \(\tau\) and \(\widetilde{\tau} = \varinjlim \widetilde{\tau}_n\), where \(\widetilde{\tau}= \tau \circ E\) (see~\cite[Proposition 4.10]{Hossain-Inductive-limits-partial}). Since each \(\alpha_n\) has the Haagerup property and \((A^{(n)}, \tau_n)\) has the Haagerup property, it follows from Proposition~\ref{prop-partial-crossed-prod-Haagerup} that \(( A^{(n)}\rtimes_{\alpha^{(n)},r}G, \widetilde{\tau}_n)\) has the Haagerup property. 
	Therefore, Proposition~\ref{prop-ind-limit-Haagerup} ensures that \((A\rtimes_{\alpha, r}G, \widetilde{\tau})\) has the Haagerup property. 
	\end{proof}


\end{document}